\documentclass{amsart}
\usepackage{amscd, amsfonts, amsmath, amsthm, amssymb}
\usepackage{pst-all}
\usepackage{tabularx}
\usepackage{ltablex}
\usepackage{longtable}
\usepackage{pstricks}
\usepackage{stmaryrd}
\usepackage[all]{xy}
\usepackage{srcltx}

\usepackage[
  backref,
  colorlinks = true
]{hyperref}

\newcommand{\Z}{\mathbb{Z}}
\newcommand{\F}{\mathbb{F}}
\newcommand{\Q}{\mathbb{Q}}

\newcommand{\C}{\mathbb{C}}

\DeclareMathOperator{\Integer}{Int}
\providecommand{\Int}[1]{\Integer \left( #1 \right)}

\providecommand{\set}[1]{ \left\{ #1 \right\} }

\theoremstyle{plain}
\newtheorem{theorem}{Theorem}[section]

\newtheorem{proposition}[theorem]{Proposition}

\theoremstyle{definition}
\newtheorem{definition}[theorem]{Definition}

\theoremstyle{remark}
\newtheorem{remark}[theorem]{Remark}

\numberwithin{equation}{theorem}
\numberwithin{table}{section}

\begin{document}

%\author[J.A. {Lara Rodr\'iguez}]{Jos\'e Alejandro {Lara Rodr\'iguez}}
\author{Jos\'e Alejandro {Lara Rodr\'iguez}}

\title{Special relations between  multizeta values  and parity results}

\address{
Facultad de Matem\'aticas de la Universidad Aut\'onoma de Yucat\'an
Anillo Perif\'erico Norte, Tablaje Cat. 13615, Colonia Chuburn\'a Hidalgo Inn,
M\'erida Yucat\'an, M\'exico
}

\address{Departamento de Control Autom\'atico, 
Centro de Investigaci\'on y de Estudios Avanzados del IPN,
Av. Instituto Polit\'ecnico Nacional 2508, San Pedro Zacatenco, 07360,
M\'exico, D.F.
}
\email{lrodri@uady.mx, jlara@ctrl.cinvestav.mx  }
\date{June 8, 2011}


\begin{abstract}
We study relations between  multizeta values for function fields 
introduced by Thakur in \cite{Thakur_Multizeta08, Thakur}. The $\F_p$-span of
 Thakur's multizeta values is an algebra \cite{Thakur_Shuffle09}. In
particular, the product $\zeta(a)\zeta(b)$ is a
linear combination of  multizeta values. 
In this paper, several of the conjectures formulated in \cite{Jalr10,
Thakur_Multizeta08} for small values or for special families of $a$ about how to
write $\zeta(a)\zeta(b)$
as an $\F_p$-linear combination of multizeta values, are proved. Also, the
parity conjecture formulated in \cite{Thakur_Multizeta08} is proved.
\end{abstract}

\maketitle

\section{Introduction}

There are various interesting analogies \cite{Goss, Thakur, Villa2006,
Rosen2002} between function
fields over finite fields and number fields. These analogies have been 
used to guess and prove results in one setting from the other. 
 We start, at a very basic level, 
with the simplest analogies. The field $K = \F_q(t)$ of rational
functions over $\F_q$ is a good analogue of the field $\Q$ of rational 
numbers. The
polynomial ring $A = \F_q[t]$ is the analogue of the ring of integers $\Z$.
Similarly, we have analogies $K_\infty \leftrightarrow \mathbb{R}$ and $\C_{\infty}
\leftrightarrow \C$, where the notation is defined below. 
 These analogies  have helped 
the development of number theory. 

Recall that the Riemann zeta function is defined as $\zeta_{\Z}(s) =
\sum_{n=1}^{\infty}n^{-s}$, where $s \in \C$ and $\Re s>1$.  There is a rich
$\emph{special values}$ theory associated to $\zeta_{\Z}(s)$, which is
intimately connected to Bernoulli numbers, $B_n$. If $n \geq 0$, we have
$\zeta_{\Z}(-n)=-{B_{n+1}}/{(n+1)}$. Consequently, if $n\ge 1$,
$\zeta_{\Z}(-2n)=0.$ Such zeros are called trivial zeros and they are simple
zeros. With respect to the non-trivial zeros, the well known Riemann hypothesis
says that the non-trivial zeros of $\zeta_{\Z}(s)$ lie on the line $\Re
s={1}/{2}$. It is still unknown whether the Riemann 
hypothesis holds. For $m=2k$, $k>0$
an integer, we have Euler's Theorem
\begin{align*}
  \zeta_{\Z}(m)=-\frac{B_m(2\pi i)^m}{2(m!)}.
\end{align*}
There is no simple formula for $\zeta_{\Z}(2k+1)$ analogous to the
previous one. It is not known even whether $\zeta_{\Z}(2k+1)$ is rational
or irrational, except for $k=1$ when it is irrational by well-known result of
Ap\'ery \cite{Apery1979}.

For function field analogy, the Artin-Weil zeta function  is defined
by $\zeta_A(s) = \sum \mbox{Norm}(I)^{-s}$, where the sum is over nonzero ideals
and $s$ is a complex variable, with $\Re s>1$. The Riemann hypothesis in this
case is known by Weil's theorem, but it is only a rational function of
$q^{-s}$. So, for example there cannot be an analogue of Euler's Theorem
connecting $\zeta_A(2k)$ to $(2\pi i)^{2k}$.

A more suitable analogue of transcendental special values of the Riemann zeta
are the  Carlitz zeta values  defined by $\zeta_A(s) = \sum
_{a\in A_+} a^{-s}$, where $s\in \Z_+$ and $A_+$ denotes the monic polynomials
in
$A = \F_q[t]$. Here the requirement monic is playing the role of ``positive'' in
the classical Riemann zeta function $\zeta_{\Z}(s)$. In other words, instead of
the norm which just depends on the degree of the polynomial, Carlitz  used the
whole polynomial, paying the price of considering a smaller domain for $s$,
since we do not know how to raise a polynomial
to a complex power. More justification lies in the following result 
%, which we state without a complete explanation
\cite{C1, C2}, \cite[Theorem
5.2.1]{Thakur}. If $m$ is `even' (that
is, a multiple of $q-1$) and positive,
\begin{align*}
  \zeta _A(m)= -B_m\widetilde{\pi}^m/(q-1)\Pi(m).
\end{align*}
Here, $B_m \in K$ is a Bernoulli analogue, $\Pi(m) \in A$
is a factorial analogue, and 
\begin{align*}
  \widetilde{\pi}=(t-t^q)^{\frac{1}{q-1}}\prod_{n=1}^\infty\left(1-\frac{t^{q^n}
-t}{t^{q^{n+1}}-t}\right)
\in \mathbb{C}_\infty, 
\end{align*}
plays the role of $2\pi i$ and is known to be transcendental over $K$
\cite{Wade1941}. There is no functional equation known. But in fact, much is
known about the nature of the special values at positive integers in contrast
to the classical case. We have the following result due to Anderson, Thakur,
and Yu \cite{A-T1, Yu1991}: For $m$ positive, $\zeta_A(m)$ is transcendental
over $K$, and $\zeta_A(m)/\widetilde{\pi}^m$ is also transcendental if $m$ is
not
`even'. 

In this case, the analogue of the Riemann hypothesis is known
(\cite{Wan, DV1, Goss96, Sheats, Thakur,  BA_DV_VS2010}). Orders of vanishing
of
zeta at
negative integers are not yet fully understood (but see \cite{Thakur, Goss96,
DV2, BADVMB}). For the details of the analytic continuation due to Goss,
the theory of special 
values, its links with cyclotomic theory, periods of $t$-motives, etc., we
refer to \cite{Goss96, Thakur}. 

Now we turn to multizeta values.
 We refer to \cite{Waldschmidt}, and references in there, for a survey of many
exciting 
recent developments related to the  multizeta values introduced by Euler and their 
connections with theory of algebraic number fields. We will be concerned 
with an analogous theory of function fields.

Dinesh Thakur \cite[Section 5.10]{Thakur} introduced two types of multizeta
values for function fields over finite fields of characteristic $p$, one
complex valued (generalizing the Artin-Weil zeta function) and the other with
values in Laurent series over finite fields (generalizing the Carlitz-Goss zeta
function). In this paper, we only focus on the latter. For its properties,
connections with Drinfeld modules and Anderson $t$-motives, we refer the reader
to \cite{A-T2, Thakur,Thakur_Multizeta08, Thakur_Shuffle09}. 

Thakur  proves the existence of ``shuffle'' relations for the
multizeta values (for a general $A$ with a rational place at infinity)
\cite{Thakur_Shuffle09}. In particular, he shows that the product
of multizeta values can also be expressed as a sum of some multizeta values, so
that the $\F_p$-span of all multizeta values is an algebra. 
In the function field case, the identities are much more complicated than the
classical shuffle identities. In fact there are two types of identities, one
with $\F_p(t)$ coefficients and the other with $\F_p$ coefficients. 
Note that although for many purposes a good analogue of $\Q$ is $\F_q(t)$, 
the prime field in  characteristic $p$  is $\F_p$ as  $\Q$ is the prime
field in characteristic $0$. We concentrate only on the latter type.

The results in \cite{Thakur_Shuffle09}, although effective, are not explicit
and bypass the explicit conjectures formulated in \cite{Thakur_Multizeta08,
Jalr, Jalr10}. 
In this paper, we use the ideas of the process in \cite{Thakur_Shuffle09} to
prove the main conjecture formulated in \cite{Jalr, Jalr10}. In this paper,
several  conjectures for small values of $a$
(Theorems~\ref{Delta1b}, \ref{thm_conjecture2.1}, \ref{thm_conjecture2.6},
and \ref{thm_conjecture2.3}), and  a conjecture for special large values of
$a$ and $b$ (Theorem~\ref{large_indices}) are proved. 
Also, the parity conjecture (Theorem~\ref{parity_conjecture}) formulated
in \cite{Thakur_Multizeta08} is proved.

\section{Notation}

\begin{tabularx}{\linewidth}{l X}
$\Z$  & Ring of integers,\\
$\Z_+$ &  positive integers,\\
$q$ &  a power of a prime $p$, $q = p^s$,\\
$\F_q$ & a finite field of $q$ elements,\\
${t}$&  an independent variable,\\
$A$ & the polynomial ring $\F_q[t]$,\\
$A_+$ &   monic polynomials in  $A$,\\
$K$ &   the function field $\F_q(t)$,\\
$K_\infty$ & $ \F_q((1/t))$, the completion of $K$ at $\infty$,\\
$\C_\infty$ & completion of an algebraic clousure of $K_\infty$,\\
$A_d$ & $ \mbox{elements in }A \mbox{ of degree }d$,\\
$A_{d^+}$ & $ A_d \cap A_+$, \\
$[n]$ &   $= t^{q^n}-t$,\\
$\ell_n$ & $=  \prod _{i= 1} ^n (t-t^{q^i})= (-1)^n [n][n-1]\dotsm
[1]$, \\
`even' & $\mbox{multiple of } q-1$,\\
$\deg$ & $\mbox{function assigning to } a\in A \mbox{ its degree in } t$,\\
$\Int{x}$       & =   $\begin{cases}
                    0 & \mbox{ if } x \mbox{ is not integer}\\
                    1 & \mbox{ if } x \mbox{ is integer}\\
                    \end{cases}$ \\
$\lfloor x \rfloor$ & the largest integer not greater than $x$.
\end{tabularx}

\section{Thakur's multizeta values}

For $s\in\Z_+$, the \emph{Carlitz zeta values} \cite{Goss96, Thakur} are
defined as
\begin{align*}
\zeta(s) = \zeta_A(s):=\sum_{a\in A_+} \frac{1}{a^s} \in K_\infty.
\end{align*}

 For $s\in \Z$ and $d\geq 0$, we write
\begin{align*}
S_d(s) := \sum _{ a\in A_{d^+}} \frac{1}{a^{s}} \in K.
\end{align*}

Given integers $s_i\in \Z_+$ and $d\ge 0$ put
\begin{align*}
  S_d(s_1,\dotsc,s_r) = S_d(s_1) \sum _{d > d_2 >\dotsb > d_r \ge 0}
S_{d_2}(s_2) \dotsm S_{d_r}(s_r) \in K.
\end{align*}

For $s_i\in \Z_+$, Thakur's  \emph{multizeta values}  \cite{Thakur,
Thakur_Multizeta08} are defined by:
\begin{align*}
\zeta(s_1,\dotsc,s_r):= \sum _{d_1>\dotsb >d_r\ge 0} S_{d_1}(s_1) \dotsm
S_{d_r}(s_r)= \sum \frac{1}{a_1^{s_1}\dotsm a_r^{s_r} }\in K_\infty,
\end{align*}
where the second sum is over all $a_i\in A_+$ of degree $d_i$ such that
$d_1>\dotsb >d_r\ge 0$. We say that this multizeta value has depth $r$ and
weight  $\sum s_i$.

\section{Relations between multizeta values}

Recall that Euler's multizeta values  $\zeta$ (only in this paragraph, the
greek letter $\zeta$ will be
used to denote the
classical multizeta values) are defined by $\zeta(s_1, \dotsc, s_r) = \sum
\left(n_1^{s_1} \dotsm n_r^{s_r} \right)^{-1}$, where the sum is over positive
integers $n_1> n_2 >\dotsb >n_r$ and $s_i$ are positive integers, with $s_1>1$
(this condition is required for convergence). Since $n_1=n_2$, $n_1>n_2$ or
$n_2>n_1$, 
we  have the ``sum shuffle
relation''
\begin{align*}
\zeta(a)\zeta(b)&= \sum _{n_1=1 } ^\infty \frac{1}{n_1^a} \sum _{n_2=1 } ^\infty
\frac{1}{n_2^b}
= \sum _{n_1=n_2}\frac{1}{n_1 ^{a+b}} + \sum _{n_1>n_2}\frac{1}{n_1^a n_2^b} +
\sum _{n_2>n_1}\frac{1}{n_2^b n_1^a }\\
&= \zeta(a+b)+\zeta(a,b)+\zeta(b,a).
\end{align*}

In the function field case, this sum shuffle relation fails because there are
many polynomials of a given degree. In contrast to the classical sum shuffle,
in the function field case, the identities we get are much more involved.

For $s_1,s_2\in \Z_+$, put
\begin{align*}
S_d(s_1,s_2)=\sum _{\substack {d=d_1>d_2\\a_i\in
A_+}}\frac{1}{a_1^{s_1}a_2^{s_2}}
\end{align*}
where $d_i=\deg(a_i)$. For $a,b\in \Z_+$, we define
\begin{align*}
\Delta_d (a,b)=S_d(a)S_d(b)-S_d(a+b).
\end{align*}
We write $\Delta(a,b)$ for $\Delta_1(a,b)$. The definition implies
that $\Delta_d(a,b)=\Delta_d(b,a)$.

The next two theorems (the second theorem in the reference has 
implications to higher genus function fields, but we state only 
a special case relevant to us) are  due to Thakur \cite[Theorems
1, 2]{Thakur_Shuffle09}. 

\begin{theorem}
   Given $a,b \in \Z_+$, there are $f_i\in \F_p$ and $a_i\in
\Z_+$, so that
\begin{align}
\Delta_d(a,b) =  \sum f_i S_d(a_i,a+b-a_i) \label{Delta1}
\end{align}
holds for $d=1$.\hfill\qed
\end{theorem}

\begin{theorem}\label{shuffle_thm2}
 Fix $A$.  If \eqref{Delta1} holds for some $f_i\in  \F_p^\times$ and distinct
$a_i \in \Z_+$
for $d=1$, then \eqref{Delta1} holds for all $d\ge 0$. In this case, we have the
shuffle relation
\begin{align*}
   \zeta(a)\zeta(b) -\zeta(a+b) -\zeta(a,b) - \zeta(b,a)= \sum f_i
\zeta(a_i,a+b-a_i).
\end{align*}
\hfill\qed
\end{theorem}

\begin{remark}
For each $a,b\in \Z_+$ the set $S(a,b)$ of pairs $(f_i,a_i)$ is
independent of $d$.  Notice that $S(a,b) = S(b,a)$.
\end{remark}

\section{The `even' restriction}

In \cite{Thakur_Multizeta08}, Thakur conjectures that in the multizeta value
identities all the iterated indices are `even' and gives some heuristics
reason for it.  The following theorem  proves the parity conjecture of
Thakur. A different proof is given in \cite{Jalr11}, but it is much 
more involved.

\begin{theorem}\label{parity_conjecture}
  Let $A = \F_q[t]$. Given $a,b \in \Z_+$, there are $f_i\in \F_p$ and $a_i\in
\Z_+$, so that
\begin{align}
S_1(a) S_1(b) - S_1(a+b) = \sum f_i S_1(a_i). \label{Delta}
\end{align}
Moreover, the $a_i$'s are such that $a+b-a_i$ are `even'.
\end{theorem}
\begin{proof}
The first part is proved in \cite{Thakur_Shuffle09}. Here, we shall prove the
second part following the proof in there.
By specializing \cite[3.3]{Thakur_Multizeta08} to $d=1$ we
see that 
\begin{align}
  S_1(k+1) =\frac{(-1)^{k+1}}{[1]^{k+1}} \left(
1 + \sum _{k_1=1}^{\lfloor k/q \rfloor} \binom{k-k_1(q-1)}{k_1}(-1)^{k_1}
[1]^{k_1(q-1)}
\right).\label{s_1(k+1)}
\end{align}
To simplify a little bit the notation, let us make the change $ U
\longleftrightarrow \frac{1}{[1]}$. Equation \eqref{s_1(k+1)} becomes
\begin{align*}
S_1(a ) & = (-1)^{a} U^{a}
\left(1 +
\sum _{i=1}^{ n_a } \binom{ a-1-i(q-1) }{i} (-1)^{i}
U^{-i(q-1)}
\right) \\
& = \alpha _{a,0} U ^{\varphi_a(0)} + \alpha_{a,1}U ^{\varphi_a(1)} +
\dotsb + \alpha_{a,n_a} U^ {\varphi_{a}(n_a)},
\end{align*}
where $\varphi_a(i) = a- i(q-1)$, $0 \le i  \le n_a = { \lfloor (a-1)/q \rfloor
}$, and
\begin{align*}
\alpha_{a,i} =  \begin{cases}
                        (-1)^a & \mbox{if } i=0, \\
\binom{ \varphi_{a-1}(i) }{i} (-1)^{a+i} & \mbox{for } i = 1,  \dotsc, n_a.
                      \end{cases}
\end{align*}

Thus, $S_1(a)$ is a $\F_p$-linear combination of powers of $U$ and, therefore,
so is $\Delta(a,b)$. More precisely, since $\varphi_a(i) +
\varphi_b(j) = \varphi_{a+b}(i+j)$, and $\alpha_{a,0}
\alpha_{b,0} = \alpha_{a+b,0} = (-1)^{a+b}$, we have
\begin{multline*}
S_1(a) S_1(b) -S_1(a+b)    = \\
\left( \sum _{i=0}^{n_a} \alpha_{a,i} U^{\varphi_a(i)} \right)
\left( \sum _{j=0}^{n_b} \alpha_{b,j} U^{\varphi_b(j)} \right)
- \sum _{l=0}^{n_{a+b}} \alpha_{a+b,l} U^{\varphi_{a+b}(l)} \\
 = \sum _{k=0}^{n_a+n_b} \beta_k U^{\varphi_{a+b}(k)}
- \sum _{l=0}^{n_{a+b}} \alpha_{a+b,l} U^{\varphi_{a+b}(l)}\\
= \sum _{k=1}^{n_a+n_b} \beta_k U^{ \varphi_{a+b}(k)}
- \sum _{k=1}^{n_{a+b}} \alpha_{a+b,k} U^{\varphi_{a+b}(k)}.
\end{multline*}
where $\beta_k = \sum_{i+j=k} \alpha_{a,i}\alpha_{b,j}$. Notice that
$$a+b - \varphi_{a+b}(k) =(a+b)-(a+b) +  k(q-1) = k(q-1).$$
Therefore, $ S_1(a) S_1(b) -S_1(a+b)$ is a polynomial in
$U$ with coefficients in $\F_p$  of degree less than $a+b$, such that
$a+b-i$  is `even' for every power $i$ of $U$. Write
\begin{align*}
   S_1(a) S_1(b) -S_1(a+b)  = \theta_n U^n + \dotsb + \theta_0 U ^0,
\end{align*}
where $\theta_n \ne 0$, $n < a+b$. Let $f_1 = (-1)^{n}\theta_n$ and $a_1 = n$.
Each power of $U$ in $S_1(a_1)$ is of the form $\varphi_{a_1}(i) = a_1 -i(q-1)$.
Since $q-1$ divides $a+b-a_1$, it divides $a+b-a_1+i(q-1) = a + b -
\varphi_{a_1}(i)$. Then, $S_1(a) S_1(b) -S_1(a+b) - f_1 S_1(a_1)$ is again a
polynomial in $U$ with coefficients in $\F_p$ of degree less than $n$, and each
power of $U$ satisfies the `even' condition. We continue in this way,
inductively, untill the sum is vacuous.
\end{proof}

\begin{theorem}
  Fix $q$. Let $K$ be a function field of one variable with field of constants
$\F_q$; let $\infty$ be a place of $K$ of degree one, and let $A$ be the ring of
elements of $K$ with no poles outside $\infty$. Given $a,b\in \Z_+$, there are
$f_i\in \F_p$ and $a_i\in  \Z_+$ such that
$$ \zeta(a)\zeta(b) -\zeta(a+b) -\zeta(a,b) - \zeta(b,a)= \sum f_i
\zeta(a_i,a+b-a_i),$$
with $a+b-a_i$ `even'.
\end{theorem}
\begin{proof}
By the above theorem, there are $f_i\in \F_p$ and $a_i\in \Z_+$ such that
\eqref{Delta} holds, and $a+b-a_i$ are `even'. By
Theorem~\ref{shuffle_thm2}, or rather by its more general version 
in \cite{Thakur_Shuffle09} we have for general $A$ as above that 
\begin{align*}
S_d(a)S_d(b)- S_d(a+b) = \sum f_i S_d(a_i,a+b-a_i)
\end{align*}
holds for all $d\ge 0$. 
\end{proof}

\section{A relation for large indices}

In this section we shall prove the Conjecture 2.8.1 formulated in
\cite{Jalr, Jalr10} in two different ways. In \cite{Jalr11}, there is a third
proof
of this result.
 
Lucas Theorem  \cite{Lucas1878a, Lucas1878, Fine1947} gives a method to
determine the value of the binomial coefficient $\binom{m}{n}$ modulo a prime
number $p$:
\begin{align*}
  \binom{m}{n} \equiv \binom{m_0}{n_0} \dotsm \binom{m_k}{n_k} \bmod p,
\end{align*}
where $m = m_0 + m_1 p + \dotsb + m_k$ and  $n = n_0 + n_1 p + \dotsb + n_k$
are the base $p$ expansions of $m$ and $n$, respectively. Since
$\binom{a}{b}=0$ if $b>a$, Lucas Theorem implies that the binomial coefficient
$\binom{m}{n}$ does not vanish modulo $p$ if and only if there is no carry
over in base $p$ in the sum of $n$ and $m-n$.

\begin{proposition}\label{prop1}
For general $q$, we have
\begin{align}
  S_1(2q^n-1)
= -\frac{[n+1]}{[1]^{2q^n}}.\label{special_case3}
\end{align}
\end{proposition}
\begin{proof}

We shall prove that for any $k_1$, $0 \le k_1 \le 2q^{n-1}-1$,
\begin{align*}
 \binom{2q^n-2-k_1(q-1)}{k_1}(-1)^{k_1} =
\begin{cases}
  1 & \mbox{if } k_1 \in \{0,1,1+q, \dotsc, 1+q+ \dotsb +q^{n-1} \}, \\
0 & \mbox{otherwise.}
\end{cases}
\end{align*}
By Lucas Theorem, only the terms where there is no carry over in base $p$ in the
sum
of $k_1$ and $2q^n-2-k_1q$ need to be considered.
The  base $p$ expansion of   $2q^{n-1}-1$ is
\begin{align*}
2q^{n-1}-1  = (p-1) + (p-1)p + \dotsb + (p-1)p^{s(n-1)-1} + p^{s(n-1)}.
\end{align*}

Let $k_1 = a_0 + a_1 p + \dotsb + a_{s(n-1)-1} p ^{s(n-1)-1} + a_{s(n-1)} p
^{s(n-1)}$ be the base $p$ expansion of $0 \le k_1 \le 2q^{n-1}-1$, where
$a_{s(n-1)} \in \{0,1\}$. Therefore, the base $p$ expansion of $k_1q$ is
$ a_0 p^s+ a_1 p^{s+1} + \dotsb + a_{s(n-1)-1} p ^{sn-1} + a_{s(n-1)} p
^{sn}$. Finally $2q^n-2 = (p-2) +   (p-1)p + \dotsb + (p-1)p^{sn-1} + p^{sn}$
 is the base $p$ expansion of $2q^n-2$.
Let $b_i$ denote the digits of $2q^n-2-k_1q$. Since the first $s$ digits  of
$2q^n-2$ are $p-2,p-1, \dotsc, p-1$, the first  $s$ digits of $k_1 q$ are zero,
and $k_1 q +(2q^n-2-k_1 q) = 2q^n-2$, it follows that the first $s$ digits of
$2q^n-2 -k_1 q$ are $b_0 = p-2$ and
$b_i = p-1$, $1 \le i \le s-1$. Next we determine the remaining $b_i$'s
assuming that there is no carry over in the sum of $2q^n-2-k_1q$ and $k_1$. By
this
assumption,  it follows immediately that $a_i=0$ for $1\le i \le s-1$. Thus,
$b_{s+1} = \dotsb = b_{2s-1} = p-1$. Using again that there is no carry over in
the sum of $k_1$ and $2q^n-2-k_1q$, we get $a_{s2-1} = \dotsb = a_{3s-1}=0$ and,
therefore, $b_{2s+1} = \dotsb = b_{3s-1}=p-1$.  By continuing  this way we
obtain
that $k_1 = a_0 + a_sp^s + a_{2s}p^{2s}+ \dotsb +
a_{s(n-1)}p^{s(n-1)}$. Now, since $a_k +b_k \le p-1$, $a_k + b_{k+s}=p-1$ and
$b_0=p-2$, we conclude  that $a_{si}=0$ or $a_{si}=1$. It follows that if
$a_{si}=0$ for some $i$, then $a_{sj}=0$ for $j>i$.

Now, if $k_1 = 0$, clearly $\binom{2q^n-2-k_1(q-1)}{k_1}(-1)^{k_1}=0$. Let us
assume that $k_1 = 1 + q + \dotsb + q^i$ for some $0 \le i \le n-1$. Then,
${q^{i+1}-1} = (1 + q + \dotsb + q^{i})(q-1)$. Therefore, $2q^n-2-k_1(q-1) =
q^n-1 + q^{i+1}(q^{n-i-1}-1)$. The digits $c_j$ of $2q^n-2-k_1(q-1)$ are
$c_{s(i+1)} = p-2$, $c_{sn} =1$, and
$c_j=p-1$, otherwise. The digits of $k_1$ are $a_j = 1$ for $j = 0,s, \dotsc,
si$ and $0$ otherwise.
By Lucas Theorem, in $\F_p$ we have
\begin{align*}
\binom{2q^n-2-k_1(q-1)}{k_1} = \prod _{j=0}^{sn} \binom{c_j}{a_j} 
=  (p-1)^{i+1} = (-1)^{i+1}.
\end{align*}
Then $\binom{2q^n-2-k_1q}{k_1}(-1)^{k_1}=(-1)^{i+1+k_1}=1$ because  $i+1$ and
$k_1$ have the same parity.
Since $(-1)^{2q^n-1}=-1$ for any $q$,   $\left \lfloor \frac{2q^n-2}{q}
\right \rfloor = 2q^{n-1}-1$, and $\sum_{i=0}^{n}[1]^{q^i} = [n+1]$, by
\eqref{s_1(k+1)} we have
\begin{align*}
S_1(2q^n-1)& = -\frac{1}{[1]^{2q^n-1}}
\left(1 +
\sum _{k_1=1}^{ 2q^{n-1}-1 } \binom{ 2q^n-2-k_1(q-1) }{k_1} (-1)^{k_1}
[1]^{k_1(q-1)}
\right) \\
& =  -\frac{1}{[1]^{2q^n-1}}
\left(
1 + [1]^{q-1} + [1]^{q^2-1} + \dotsb + [1]^{q^n-1}
\right) \\
 & = -\frac{[n+1]}{[1]^{2q^n}}.
\end{align*}
\end{proof}

\begin{remark}
  The Proposition~\ref{prop1} proves the Conjecture 2.10 in \cite{Jalr10} for
the case $m=2$ and $d=1$. 
\end{remark}

Next we prove the Conjecture 2.8.1 formulated in \cite{Jalr10}.
For general $q$ we have the following theorem.
\begin{theorem}\label{large_indices}
Let $q$ be arbitrary. Then
\begin{align*}
  \zeta(q^n) \zeta(q^n-1) = \zeta(2 q^n-1) + \zeta(q^n-1,q^n).
\end{align*}
\end{theorem}

\begin{proof}[Proof 1]
By Theorem~\ref{shuffle_thm2} it is enough to prove
\begin{align}
   S_1(q^n)S_1(q^n-1) = S_1(2q^n-1) - S_1(q^n). \label{Conjecture281}
\end{align}
From \cite[3.3]{Thakur_Multizeta08}, we know that
\begin{align}
 S_d(ap^n) & =  1/{\ell _d ^{ap^n} },\quad \mbox {if }a\le
q,\label{special_case1} \\
S_d(q^i-1) & =\frac{\ell _{d+i-1} }{\ell_{i-1}\ell_d ^{q^i}
}=\frac{[d+i-1]\dotsm
[d+1] }{[i-1]\dotsm [1]\ell _d ^{q^i-1} }.\label{special_case2}
\end{align}
Equation \eqref{Conjecture281} follows from formulas \eqref{special_case3},
\eqref{special_case1}, \eqref{special_case2},  and from
 the fact $[n+1]-[n] = [1]^{q^n}$.
\end{proof}

\begin{proof}[Proof 2.]
By definition of $\Delta(a,b)$, it follows that
\begin{align}
  \Delta(a,b) 
= \sum _{ \substack{n_1\ne n_2 \\ n_1,n_2\in A_{1^+}} } \frac{1}{n_1^a n_2^b}
= \sum _{n_2\in A_{1^+}} \frac{1}{n_2^b}
\left(
S_1(a) - \frac{1}{n_2^a}
\right).\label{Delta_ab_in_terms_of_S1a}
\end{align}
For any $a\in A_{1^+}$, we have $[1]/a = a^{q-1}-1$. By
\eqref{special_case1}, we also have $[1]^{q^n} S_1(q^n) =-1$. Taking $b = q^n-1$
in \eqref{Delta_ab_in_terms_of_S1a}, we have
\begin{align*}
  \Delta(q^n,q^n-1)  &=S_1(q^n) \sum _{a \in A_{1^+}} \frac{1}{a^{q^n-1}}
\left(
1 + \frac{[1]^{q^n}}{a^{q^n}}
\right) \\
& = S_1(q^n) \sum _{a \in A_{1^+}} \frac{1}{a^{q^n-1}} a^{(q-1)q^n}\\
& = S_1(q^n) S_1(-N),
\end{align*}
where $N = {q^{n+1}-2q^n+1}$. To finish, we prove that $S_1(-N)=-1$. Now
\begin{align*}
  S_1(-N)  
 = \sum _{\theta \in \F_q} (t+\theta)^N 
 = t^N +  \sum _{\theta \in \F_q^*} (t+\theta)^N
 = t^N +  \sum _{l=0}^N \binom{N}{l} t^{N-l}  \sum _{\theta \in \F_q^*}
\theta^l.
\end{align*}

Since the sum $\sum_{\theta\in \F_q^*} \theta^l$ is 0 if $q-1$ does not
divide $l$, and $-1$ if $l\ge 0$ is divisible by $q-1$, and $N \equiv 0
\bmod (q-1)$, we get
\begin{align*}
  S_1(-N) = -1- \sum _{0<l<\frac{N}{q-1}} \binom{N}{l(q-1)} t^{N-l(q-1)}.
\end{align*}
Let $m=sn$. The base $p$ expansion of $N$ is $N = 1 + (p-2)p^m + (p-1)p^{m+1}
+\dotsb + (p-1)p^{m+s-1}$.
Let $l(q-1) = \sum _{k=0}^{m+s-1} b_k p^k$ be the base $p$ expansion of
$l(q-1)$. By Lucas Theorem, the following equality holds in $\F_p$
\begin{align*}
\binom{N}{l(q-1)} = 
\binom{1}{b_0}
\binom{0}{b_1}
\dotsm
\binom{0}{b_{m-1}}
\binom{p-2}{b_m}
\binom{p-1}{b_{m+1}}
\dotsm
\binom{p-1}{b_{m+s-1}}.
\end{align*}
Therefore, $\binom{N}{l(q-1)} \ne 0$ if and only if  $b_0\le 1$, $b_k=0$ for
$k=1,\dotsc,m-1$, and $b_m \le p-2$.

Since $p^m$ and $q-1$ are coprime, $p^m$ is a unit in $\Z/(q-1)\Z$. Therefore,
the equation $j \equiv -\iota p^m\bmod (q-1)$ has always a solution.
Furthermore, there is exactly one solution in the range $0 \le \iota <q-1$. For
$j=0$, the solution is $0$, and for $j=1$ the solution is $q-2$ because $N = 1
+(q-2)p^m \equiv 0\bmod(q-1)$.

If $\binom{N}{l(q-1)}\ne 0$, then $l(q-1) = j + i p^m$, $0\le j \le 1$, $0 \le i
\le q-2$, where $i = \sum _{k=0}^{m+s-1} b_{m+k}p^{k}$. Then, $l(q-1) \in
\set{0,N}$. Thus, $\binom{N}{l(q-1)}=0$ for $1 \le l < N/(q-1)$. Therefore,
$S_1(-N) =  -1$.
\end{proof}

\section{Relations for small values of \texorpdfstring{$a$}{a}}

In this section we  prove that for $1\le a \le p$, the sets $S(a,b)$ can
be found recursively; we also prove some of the conjectures for low values of
$a$ given in \cite{Thakur_Multizeta08, Jalr10}.

The following results will be used in this section.

\begin{theorem}\label{thm251}
Let $k$ be a positive
integer. Let $\ell(k)$ be the sum of digits of $k$ in base $q$.
\begin{itemize}
  \item [a)] If $d>\ell(k)/(q-1)$, then $S_d(-k)=0$ (\cite[Theorem
5.1.2]{Thakur}; see also \cite[Lemma~7.1]{Lee1943} and
\cite[Corollary~2.12]{Gekeler1988}).
\item [b)] $S_d(-(q^{k+d}-1)) =  (-1)^d D_{d+k}/L_d D_k^{q^d}$, where 
$D_n = [n][n-1]^q \dotsm [1]^{q^{n-1}}$ and $L_n = [n][n-1]\dotsm [1]$ 
for $n>0$,  $D_0 = 1$, and $L_0=1$ (\cite{Carlitz4};
see also \cite[Theorem~4.1]{Gekeler1988}).
\end{itemize}
\end{theorem}

In particular, when $d = 1$, we have $S_1(-(q-1))=-1$.

\begin{definition}\label{main_definitions} Let $a\in \Z_+$.
\begin{enumerate}
\item  We set
    \begin{align*}
    r_a := (q-1)p^m,
    \end{align*}
    where $m$ is the smallest integer such that $a\le p^m$.
\item For $i,j$, put
    \begin{align*}
    \phi(i,j)&:=r_a-a-j(q-1)+ir_a,\\
    \phi(j)& :=\phi(0,j).
    \end{align*}
\item We define
    \begin{align*}
    j_{a, \max}=\left \lfloor \frac{r_a-a}{q-1} \right \rfloor,
    \end{align*}
where $\lfloor x \rfloor$ is the largest integer not greater than $x$.

\item Let  $\Int{x}$ be 1 if $x$ is integer, and $0$, otherwise.

\end{enumerate}
\end{definition}

\begin{theorem}[Conjecture 4.3.1, \cite{Thakur_Multizeta08}]\label{Delta1b}
For general $q$, we have,
\begin{align*}
\zeta(1)\zeta(b)=\zeta(1+b)+ \zeta(1,b) + \zeta(b,1) +\sum _{i=0}^{\sigma-1}
\zeta(b-\phi(i),1+\phi(i)),
\end{align*}
where $\phi(i)$ and $r_1$ are as in Definition~\ref{main_definitions}, and
$b=r_1\sigma+\beta,0<\beta\le r_1$.
\end{theorem}
\begin{proof}
Let $a=1$. Then, $r_1 = q-1$. By \eqref{special_case2}, we have $[1] S_1(1) =
-1$. Proceeding as in the second proof of Theorem~\ref{large_indices}, we get
$\Delta(1,b+r_1)  = S_1(1) S_1(b)$. Therefore, $\Delta(1,b+r_1) - \Delta(1,b) =
S_1(1+b)$. Now, $\Delta(1,b)=0$ for $1\le b\le q-1$ because $1+b\le q$
\cite[Theorem 1]{Thakur_Multizeta08}. Therefore, for any $b\in \Z_+$,
$\Delta(1,b) = \sum _{i=0}^{\sigma-1} S_1(b-\phi(i))$. Theorem follows
from applying
Theorem~\ref{shuffle_thm2}.
\end{proof}

\begin{remark}
\begin{enumerate}
  \item Theorem~\ref{Delta1b} is proved for $q=2$ and $b$ arbitrary, and also
for general $q$ and $b$ `even' in \cite{Thakur_Multizeta08}.
\item The proof of Theorem~\ref{Delta1b} shows that the sets $S(1,b)$ can be
found recursively, with recursion length $r_1 = q-1$.
\end{enumerate}
\end{remark}

\begin{theorem}\label{a_between_2_and_p}
Let $a,b\in \Z_+$. Let $r_a$ be as in
Definition~\ref{main_definitions}. If $2 \le a \le p$, then 
\begin{align*}
  \Delta(a,b+r_a)- \Delta(a,b)  = \sum _{j=1}^{p-a} f_{a,j}
S_1(a+b+(p-j)(q-1)) +S_1(a+b),
\end{align*}
where $f_{a,j} = \binom{a+j-1}{j}= \binom{a+j-1}{a-1}$ is nonzero for any $j$.
In particular, the sets $S(a,b)$, $2\le a \le p$, can be found recursively with
recursion length $r_a$, by
\begin{align*}
  S(a,b+r_a) = S(a,b)\cup T(a,b+r_a),
\end{align*}
where
\begin{align*}
  T(a,b+r_a) = \set{(1,a+b)} \cup \set{ \left( f_{a,j}, a+b + (p-j)(q-1)
\right) \mid 1
\le j \le p-a}.
\end{align*}
\end{theorem}
\begin{proof}
For each $n \in A_{1^+}$, let $g_n=  (-1)^{a+1}{[1]^{p-a}}/{n^{p-a}}$. Then
\begin{align*}
 g_n= (-1)^{a+1}(n^{q-1}-1)^{p-a} 
=1+  \sum _{j=1}^{p-a} f_{a,j} n^{j(q-1)},
\end{align*}
where $f_{a,j} = \binom{p-a}{j}(-1)^{j}$. Let $\Sigma = \sum _{j=1}^{p-a}
f_{a,j} S_1(b+r_a-\phi(p-j))$. Then,
\begin{align*}
 \Sigma   =\sum _{n_2\in A_{1^+}} \sum_{j=1}^{p-a}
\frac{f_{a,j}}{n_2^{b+r_a-(r_a-a -(p-j)(q-1))}} 
 &= \sum _{n_2\in A_{1^+}} \frac{1}{n_2^{b+r_a}} \frac{1}{n_2^{a}} \sum
_{j=1}^{p-a} f_{a,j} n_2^{j(q-1)}.
\end{align*}
By \eqref{special_case2}, we have $[1]^a S_1(a) = (-1)^a$. For any $n \in
A_{1^+}$,
\begin{align*}
S_1(a) - \frac{g_n}{n^a}
= S_1(a) + S_1(a) \frac{[1]^p}{n^p}
= S_1(a) 
\left(
1 + \frac{[1]}{n}
\right)^p
= S_1(a)n^{p(q-1)}.
\end{align*}
Then, 
\begin{align*}
\Delta(a,b+r_a) - \Sigma & = \sum _{n_2\in A_{1^+}} \frac{1}{n_2^{b+r_a}}
\left(
S_1(a) - \frac{1}{n_2^a} \left(1 +\sum _{j=1}^{p-a} f_{a,j}n_2^{j(q-1)}  \right)
\right)\\
& = S_1(a)\sum _{n_2\in A_{1^+}} \frac{n_2^{p(q-1)}}{n_2^{b+r_a}}\\
& = S_1(a)S_1(b).
\end{align*}
Therefore, $\Delta(a,b+r_a) -\Delta(a,b) = \Sigma + S_1(a+b)$. 
This shows that $T(a,b+r_a)$ is exactly as claimed.  Note that
$f_{a,j} \ne 0$ because of  Lucas Theorem. Finally,
\begin{align*}
  f_{a,j} & = \binom{p-a}{j} (-1)^j
= (-1)^j \frac{1}{j!} \prod_{i=0}^{j-1} (p-a-i)
= \frac{1}{j!} \prod_{i=0}^{j-1} (a+i)
= \binom{a+j-1}{j}.
\end{align*}
\end{proof}

Next, we apply Theorem~\ref{a_between_2_and_p} to $a = 2$ and $p=2$.

\begin{theorem}[Conjecture 2.1, \cite{Jalr10}]\label{thm_conjecture2.1}
Let $q$ be a power of $p=2$. Let $r_2$, $\phi(i,j)$ and $\Int\cdot$ be as in
Definition~\ref{main_definitions}. Write $b=r_2\sigma+\beta$, $0<\beta\le r_2$.
Then
\begin{multline*}
\zeta(2)\zeta(b)=\zeta(2+b) + \zeta(2,b) + \zeta(b,2) \\ + \sum
_{i=0}^{\sigma-1} \zeta(b-\phi(i,0),2+\phi(i,0) )+
\Int{\frac{b}{q-1}}\frac{b}{q-1} \zeta(2,b).
\end{multline*}
\end{theorem}
\begin{proof}
By Theorem~\ref{shuffle_thm2}, it is enough to prove $\Delta(2,b)=D(2,b)$,
where
\begin{align}
D(2,b)  = \sum _{i=0}^{\sigma-1} S_1(b-\phi(i,0))+
\Int{\frac{b}{q-1}}\frac{b}{q-1} S_1(2).\label{D2b_for_p_equals_2}
\end{align}
Note that $b + r_2 = r_2(\sigma+1) + \beta$; also $b+r_2-\phi(i,0) = b
-\phi(i-1,0)$. It follows that
\begin{align*}
D(2,b+r_2) 
 = \sum _{i=-1}^{\sigma-1} S_1(b-\phi(i,0))+
\Int{\frac{b+r_2}{q-1}}\frac{b+r_2}{q-1} S_1(2).
\end{align*}
Since $q-1$ divides $r_2$,  then $q-1$ divides $b$ if and only if $q-1$ divides
$b+r_2$; also $b-\phi(-1,0)=2+b$. Thus, $D(2,b+r_2) - D(2,b)= S_1(2+b)$. On the
other hand, since $a = p$, by Theorem~\ref{a_between_2_and_p}, we have
$\Delta(2,b+r_2) - \Delta(2,b) = D(2,b+r_2)-D(2,b)$. 

Let us compute $\Delta(2,b)$ for $1\le b \le r_2$. Proceeding as in
the proof of Theorem~\ref{large_indices}, we get $\Delta(2,b)  =  S_1(2)
S_1(-(r_2-b))$. If $b = r_2=2(q-1)$, then $S_1(0)=0$. If $b = q-1$, then, by
Theorem~\ref{thm251} (b), $S_1(-(q-1))=-1$ follows. Suppose $1 \le b<q-1$. Then
$r_2 -b = (q-2-b) +q$ is the base $q$ expansion of $r_2-b$. Since $\ell(r_2-b)
<q-1$, by Theorem~\ref{thm251} (a), it follows that $S_1(-(r_2-b))=0$. If now,
$q-1
< b <r_2$, write $b = (q-1) + \rho$, where $0 < \rho < q-1$. Then $r_2
-b = q-1-\rho$ is the base $q$ expansion of $r_2-b$. Applying
Theorem~\ref{thm251} (a) again, we have that $S_1(-(r_2-b))=0$. In summary,
$\Delta(2,b) =S_1(2)$ if $b=q-1$ and $0$, otherwise. Therefore, $\Delta(2,b) =
D(2,b)$ for $b = 1,\dotsc, r_2$. 
\end{proof}

\begin{theorem}[Conjecture 2.6, \cite{Jalr10}]\label{thm_conjecture2.6}
Let $p$ be any prime and let $q$ be a power of $p$. Then
\begin{multline*}
\zeta(2)\zeta(b)=\zeta(2+b) + \zeta(2,b) + \zeta(b,2) \\
+ \sum _{j=0}^{p-1} (j+2)\sum _{b-\phi(i,p-1-j)>2}
\zeta(b-\phi(i,p-1-j),2 + \phi(i,p-1-j) )\\
+ \Int{\frac{b}{q-1}}\frac{b}{q-1}
\zeta(2,b), 
\end{multline*}
where $\phi(i,k)$ and $\Int\cdot$ are as in Definition~\ref{main_definitions}. 
\end{theorem}
\begin{proof}
Let $f_b = \Int{\frac{b}{q-1}}
\frac{b}{q-1}$.  By Theorem~\ref{shuffle_thm2}, it is enough to prove
$\Delta(2,b) = D(2,b)$, where 
\begin{align}
 D(2,b) = \sum _{j=0}^{p-1} (j+2) \sum _{i=0}^{n(b,j)}
S_1(b+2-(pi+1+j)(q-1)) + f_b S_1(2).\label{D(2,b) ver1}
\end{align}
We have $b+r_2+2 -(pi+1+j)(q-1) = b+2 - (p(i-1)+1+j)(q-1)$ and $n(b+r_2,j) =
n(b,j)+1$. Then
\begin{align*}
  D(2,b+r_2) 
 = \sum_{j=0}^{p-1}(j+2) \sum _{i=-1}^{n(b,j)}
S_1(b+2-(pi+1+j)(q-1))+f_{b+r_2} S_1(2).
\end{align*}

Now,   $q-1$ divides $b$ if and only if $q-1$ divides
$b+r_2$ because  $q-1$ divides $r_2$. Therefore,
\begin{align*}
D(2,b+r_2) - D(2,b) 
 = \sum _{j=1}^{p}(j+1) S_1(2+b+(p-j)(q-1))).
\end{align*}
On the other hand, $f_{2,j} =\binom{j+1}{1}=j+1$ for $1\le j \le p-2$, $j+1=0$
when $j = p-1$, and $j+1= 1$ when $j=p$. By Theorem~\ref{a_between_2_and_p},
it follows that
\begin{align*}
  \Delta(2,b+r_2) - \Delta(2,b) & = 
\sum _{j=1}^{p-2} f_{2,j}S_1(2+b+(p-j)(q-1)) + S_1(2+b)\\
& = D(2,b+r_2) - D(2,b).
\end{align*}
To finish, we  prove that $\Delta(2,b) = D(2,b)$ for $1 \le b \le r_2$.
Let $1 \le b \le r_2$. If $n(b,j)\ge 1$, then $r_2 \ge b \ge r_2 +
(1+j)(q-1)+1$ which implies the contradiction $0 \ge (1+j)(q-1)+1 \ge 1$. Thus,
$n(b,j)<1$. Now, $n(b,j)=0$ if and only if $b-1 -(1+j)(q-1) \ge 0$. Thus,
equation~\eqref{D(2,b) ver1}  becomes
\begin{align}
  D(2,b) = \sum _{j=0}^ {\left\lfloor \frac{b-q}{q-1}\right\rfloor   }
(j+2)S_1(b+2-(1+j)(q-1)) +  f_b S_1(2).\label{D(2,b)}
\end{align}
Write $b = \lambda (q-1) + \rho$, where $0 \le \rho <q-1$. If $\rho = 0$, take
$l
= \lambda-1$; if $\rho>0$, let $l = \lambda$. Thus,  $l$ is the integer in the
set $\set{0,1, \dotsc,p-1}$ such that $l(q-1)+1
\le b \le (l+1)(q-1)$. Then, $(l-1)(q-1) \le b-q \le l(q-1)-1<l(q-1)$.
Therefore, $\left \lfloor \frac{b-q}{q-1} \right \rfloor = l-1$. Now, we
rewrite equation~\eqref{D(2,b)} as follows:
\begin{align*}
  D(2,b) & = \sum _{n \in A_{1^+}} \sum _{j=0}^{l-1}
\frac{j+2}{n^{b+2-(1+j)(q-1)}} + \sum _{n\in A_{1^+}} \frac{f_b}{n^2} \\
& = \sum _{n\in A_{1^+}} \frac{1}{n^{b+2-(q-1)}}
\left(
\sum _{j=0}^{l-1} \frac{(j+2)n^{b+2-(q-1)}}{n^{b+2-(1+j)(q-1)}} + f_b
n^{b-(q-1)}
\right)\\
& = \sum _{n \in A_{1^+}} \frac{P_n}{n^{b+2-(q-1)}},
\end{align*}
where $P_n = \displaystyle \sum _{j=0}^{l-1}(j+2)n^{j(q-1)} + f_b n^{b-(q-1)}$.
By \eqref{special_case1}, $S_1(2) = 1/[1]^2$. Then
\begin{align*}
 \Delta(2,b) - D(2,b)  
 = \sum _{n \in A_{1^+}} \frac{n^2-[1]^2-n^{q-1}[1]^2P_n}{n^{b+2}[1]^2}.
\end{align*}

Now, we compute $(n^{q-1}-1)^2 P_n$. Firstly, note
that
\begin{align*}
  n^{2(q-1)} \sum _{j=0}^{l-1}(j+2)n^{j(q-1)} & =
\sum _{j=0}^{l-1}(j+2)n^{(j+2)(q-1)},\\
-2n^{q-1}\sum _{j=0}^{l-1}(j+2)n^{j(q-1)} 
& = - \sum _{j=-1}^{l-1} 2(j+3)n^{(j+2)(q-1)},\\
\intertext{and}
\sum _{j=0}^{l-1}(j+2)n^{j(q-1)} & = \sum _{j=-2}^{l-1}  (j+4)n^{(j+2)(q-1)}.
\end{align*}
Therefore, $(n^{q-1}-1)^2 P_n$ equals
\begin{align*}
2 - n^{q-1} - (l+2)n^{l(q-1)} + (l+1)n^{(l+1)(q-1)} +  (n^{q-1}-1)^2 f_b
n^{b-(q-1)}.
\end{align*}
Since $n^{q-1}[1]^2 P_n = n^{q+1}(n^{q-1}-1)^2 P_n$, it follows that
$n^{q-1}[1]^2 P_n$ equals
\begin{align*}
 2 n^{q+1} - n^{2q} - (l+2)n^{(l+1)(q-1)+2}  + (l+1)n^{(l+2)(q-1)+2} +
n^{q-1}[1]^2 f_b n^{b-(q-1)}.
\end{align*}
Using $n^2 - [1]^2 = -n^{2q} + 2n^{q+1}$, we get
\begin{multline*}
  n^2 - [1]^2 - n^{q-1}[1]^2 P_n = (l+2)n^{(l+1)(q-1)+2} \\ -
(l+1)n^{(l+2)(q-1)+2}  - n^{q-1}[1]^2 f_b n^{b-(q-1)}.
\end{multline*}
Dividing $n^2 - [1]^2 - n^{q-1}[1]^2 P_n$ by $n^{b+2}$ and
summing over $n\in A_{1^+}$, we get
\begin{align*}
  [1]^2 \left( \Delta(2,b) - D(2,b) \right) =
(l+2)S_1(-k_1) - (l+1) S_1(-k_2) - f_b.
\end{align*}
where $k_1 = (l+1)(q-1)-b$ and $k_2 = (l+2)(q-1)-b$.
Now, if $b \equiv 0 \bmod (q-1)$, then $l = \frac{b}{q-1}-1$; thus,
$k_1 = 0$ and $k_2 =q-1$. By Theorem~\ref{thm251} (b), we have $S_1(-(q-1))=-1$;
since $S_1(0) = 0$  and $f_b =\frac{b}{q-1} = l+1$, it follows that 
$\left( \Delta(2,b) - D(2,b) \right)=0$. Suppose $b \not \equiv 0 \bmod (q-1)$.
Then, $f_b =0$. In this case, $b = l(q-1) + \rho$ with $0 < \rho <q-1$.
Therefore, $k_1 = q-1-\rho$ and $k_2 = (q-2-\rho)+q$. Note that these are the
base $q$ expansions of $k_1$ and $k_2$ because $q-1-\rho<q-1$ and
$q-2-\rho<q-1$. Then, $\ell(k_1)  =  \ell(k_2) = q-1-\rho$. Since $\ell(k_1) =
\ell(k_2) < (q-1)$, by Theorem~\ref{thm251} (a), it follows that $S_1(-k_1) =
S_1(-k_2)=0$. Therefore, $\Delta(a,b) - D(2,b)=0$ for $1 \le b \le r_2$. 
\end{proof}

\begin{remark}
Theorem~\ref{thm_conjecture2.6} generalizes Theorem~\ref{thm_conjecture2.1}.
Let $q$ be a power of $p$. Let $b\in \Z_+$  and $0\le j \le p-1$. Since
$\phi(i,p-1-j) = (pi+1+j)(q-1)-2$, the condition $b - \phi(i,p-1-j) > 2$ is
equivalent to
$n(b,j) \ge i$, where $n(b,j) = \left \lfloor \frac{b-1-(1+j)(q-1))}{r_2} \right
\rfloor$. More precisely, $b - \phi(i,p-1-j) > 2  \Longleftrightarrow
\frac{b-1}{q-1} \ge pi+1+j$. Now, we specialize to  $p = 2$ and write $b = r_2
\sigma + \beta$, $0<\beta
\le r_2$. Then, $n(b,p-1)  = n(b,1) = \sigma -1$. Equation~\eqref{D(2,b) ver1}
becomes equation~\eqref{D2b_for_p_equals_2}:
\begin{multline*}
\sum _{j=0}^{1} (j+2)\sum _{i=0} ^{\sigma-1}
S_1(b-\phi(i,1-j)) + \Int{\frac{b}{q-1}}\frac{b}{q-1} S_1(2)
\\
= 3\sum _{i=0} ^{\sigma-1}
S_1(b-\phi(i,0)) + \Int{\frac{b}{q-1}}\frac{b}{q-1} S_1(2).
\end{multline*}
\end{remark}

\begin{theorem}[Conjecture 2.3, \cite{Jalr10}]\label{thm_conjecture2.3}
Let $q=2$.  For $b\in \Z_+$, we have
\begin{multline*}
\zeta(3)\zeta(b) = \zeta(b+3) + \zeta(3,b) + \zeta(b,3) \\
+ \sum _{i=0}^{\left \lfloor (b-5)/4 \right \rfloor }
\zeta(b-1-4i, 4+4i) 
+\sum _{i=0}^{\left \lfloor (b-4)/4 \right \rfloor } \zeta(b-4i, 3+4i)\\
+ \sum _{i=1}^2 \Int{\frac{b-i}{r_3}} (\zeta(2,b+1)+\zeta(3,b)).
\end{multline*}
\end{theorem}
\begin{proof}
It is enough to prove that $\Delta(3,b) = D(3,b)$, where
\begin{multline*}
D(3,b)=\sum _{i=0}^{\left \lfloor (b-5)/4 \right \rfloor }
S_1(b-1-4i) 
+\sum _{i=0}^{\left \lfloor (b-4)/4 \right \rfloor } S_1(b-4i)\\
+ \sum _{i=1}^2 \Int{\frac{b-i}{r_3}} (S_1(2)+S_1(3)).
\end{multline*}
Using that $\lfloor x+1 \rfloor = \lfloor
x \rfloor +1$, we get $\lfloor
(b-5)/4
\rfloor = \lfloor (b-1)/4 \rfloor - 1$ and $\lfloor
(b-4)/4 \rfloor = \lfloor b/4 \rfloor -1$. Then,
\begin{align*}
  \sum _{i=0}^{\lfloor (b-1)/4 \rfloor } S_1(b+4 -1-4i)   
= \sum _{i=-1}^{\lfloor (b-5)/4  \rfloor } S_1(b -1-4i),
\end{align*}
and
\begin{align*}
\sum _{i=0}^{ \lfloor b/4 \rfloor } S_1(b+4-4i) 
= \sum _{i=-1}^{\lfloor (b-4)/4 \rfloor } S_1(b-4i).
\end{align*}
Since $(b-i) \equiv 0 \bmod 4$ if and only if $(b+4-i) \equiv 0 \bmod 4$, it
follows that $D(3,b+4) - D(3,b)  = S_1(b+3) + S_1(b+4)$.
By \eqref{s_1(k+1)}, we have
\begin{align*}
  S_1(3) =\frac{1}{[1]^3} + \frac{1}{[1]^2} = \frac{1}{n^3(n+1)^3} +
\frac{1}{n^2(n+1)^2}.
\end{align*}
A straight-forward calculation shows that $S_1(3) - \frac{1}{n^3} - n^4 S_1(3)
=1$.
It follows that
\begin{align*}
  \Delta(3,b+4) - \Delta(3,b) 
& = \sum _{n\in A_{1^+}} \frac{1}{n^{b+4}}
\left(
S_1(3) - \frac{1}{n^3} - n^4 S_1(3) - n
\right)\\
& = S_1(b+4) + S_1(b+3).
\end{align*}
Then, $\Delta(3,b)$ can be found recursively with recursion length $r_3 = 4$,
and $\Delta(3,b+4)- \Delta(3,b) = D(3,b+4)-D(3,b)$. Finally, a  direct
calculation shows that $\Delta(3,b) = D(3,b)$ for $b = 1,2,3,4$. 
\end{proof}

\begin{remark}
  Let $q = 2$ and $a = 3$. Then, $\phi(i,j) = 1-j+4i$. The condition
$b- \phi(i,0)>3$ is equivalent to $i \le (b-5)/4$. Since $j_{3,\max} = 1$, the
condition $b-\phi(i,j_{3,\max})>3$ is equivalent to $ i\le (b-4)/4$. Therefore,
the proof of Theorem~\ref{thm_conjecture2.3} confirms Conjecture 2.3 of
\cite{Jalr10}.
\end{remark}

{\bf Acknowledgments}
This work has been developed under
the direction of Dr. Dinesh S. Thakur at the University of Arizona and Dr.
Gabriel D. Villa Salvador at the Departamento de Control Autom\'atico of the
Centro de Investigaci\'on y de Estudios Avanzados del IPN (Cinvestav-IPN) in
M\'exico City. I thank Javier Diaz-Vargas for his suggestions and advice. I
want
to express my gratitude to the Universidad Aut\'onoma de Yucat\'an and the
Consejo Nacional de Ciencia y Tecnolog\'ia for their financial support. I thank
 the contributors of the Sage project~\cite{Sage444} for providing the
development environment for this research.

%%%%%%%%%%%%%%%%%%%%%%%%%%%%%%%%%%%%%%%%%%%%%%%%%%%%%%%%%%%%%%%%%%%
%                            REFERENCES                           %
%%%%%%%%%%%%%%%%%%%%%%%%%%%%%%%%%%%%%%%%%%%%%%%%%%%%%%%%%%%%%%%%%%%

%\bibliography{/home/lrodri/Dropbox/maestria/MasterThesis/references}
%\bibliographystyle{alpha}

%%%%%%%%%%%%%%%%%%%%%%%%%%%%%%%%%%%%
%for the thebibliography environment
\newcommand{\etalchar}[1]{$^{#1}$}
\def\cprime{$'$} \def\cprime{$'$}
%%%%%%%%%%%%%%%%%%%%%%%%%%%%%%%%%%%%

\end{document}